\documentclass[11pt]{article}
\usepackage{amsmath, amsthm, amsfonts, amssymb}
\usepackage{fontenc,graphicx,pstricks,psfrag}
\usepackage{comment}

\newtheorem{thm}{Theorem}[section]
\newtheorem{lem}[thm]{Lemma}

\newtheorem{dfn}[thm]{Definition}
\newtheorem{cor}[thm]{Corollary}

\theoremstyle{remark}
\newtheorem{rem}[thm]{Remark}

\newcommand{\EMD}{\E( \check{Y}^{(n)}_{\tau_n})}

\newcommand{\EMDA}{\E( \check{Y}_{\xi_0})}
\newcommand\Var{{\mathrm {Var}}}

\newcommand{\BBe}{\mathbb{E}}
\newcommand{\BBp}{\mathbb{P}}
\newcommand{\BBn}{\mathbb{N}}
\newcommand{\BBr}{\mathbb{R}}

\newcommand{\BBone}{\mbox{1}\kern-.25em \mbox{I}}

\newcommand{\vite}{a_n}
\newcommand{\usdp}{\frac{1}{2\pi}}
\newcommand{\ipp}{\int_{-\pi}^{\pi}}
\newcommand{\br}{\mathcal{B}}
\newcommand{\Int}{\mathrm{int}\,}
\newcommand{\Clo}{\mathrm{clo}\,}

\def\Horj{ \widehat{\widetilde{X^{(n)}_{u,j}}}}
\def\Hor{ \widehat{\widetilde{X^{(n)}_{u}}}}
\def\Horsn{ \widehat{\widetilde{X_{u}}}}

\def\E{\ensuremath{\mathbb{E}}}

\def\R{\ensuremath{\mathbb{R}}}
\def\N{\ensuremath{\mathbb{N}}}

\def\P{\ensuremath{\mathbb{P}}}

\def\Brp{\Big)}
\def\Blp{\Big(}
\def\brp{\big)}
\def\blp{\big(}

\def \be{\begin{eqnarray*}}
\def \ee{\end{eqnarray*}}
\def \ben{\begin{eqnarray}}
\def \een{\end{eqnarray}}

\def\QED{\hfill\vrule height 1.5ex width 1.4ex depth -.1ex \vskip20pt}

\def\nor01{\mathcal{N}(0,1)}

\def\Exp1{\mathcal{E}(1)}

\def\u01{\mathcal{U}([0,1])}

\def\QED{\hfill\vrule height 1.5ex width 1.4ex depth -.1ex \vskip20pt}

\newcommand{\ind}[1]{{\bf 1}\hspace{ -0.14 true cm} {\tt l}_{#1}}

\begin{document}
\begin{center}
{\Large {\bf Conditional large and moderate deviations for sums of discrete random
  variables. Combinatoric applications.}}
\end{center}

\vskip 6mm

\begin{center}
Fabrice GAMBOA\footnote{IMT, Equipe de Statistique et Probabilit\'es,
UMR 5219,
Universit\'e Paul Sabatier,
118 Route de Narbonne
31062 Toulouse cedex 4, France.
\texttt{gamboa@math.ups-tlse.fr}}, Thierry KLEIN\footnote{ 
IMT, Equipe de Statistique et Probabilit\'es,
UMR 5219,
Universit\'e Paul Sabatier,
118 Route de Narbonne
31062 Toulouse cedex 4, France.
\texttt{tklein@math.ups-tlse.fr}} and Cl\'ementine PRIEUR\footnote{
IMT, Equipe de Statistique et Probabilit\'es,
UMR 5219,
INSA Toulouse, GMM,
135 avenue de Rangueil,
31077 Toulouse cedex 4, France.
\texttt{clementine.prieur@insa-toulouse.fr}}
\end{center}

\vskip 6mm

{\bf Running title.} Conditional large and moderate deviations.

\vskip 6mm

{\bf Abstract.} 
We prove large and moderate deviation principles for the distribution of an empirical mean conditioned by the value of the sum of discrete i.i.d. random variables. Some applications for combinatoric problems are discussed.
\vskip 6mm


\vskip 6mm

\section{Introduction}
\label{sintro}
In many random combinatorial problems, the distribution of the interesting statistic is the law of an empirical mean built on an independent and identically
distributed
(i.i.d.) sample conditioned by
some exogenous integer random variable (r.v.). In general, 
this exogenous r.v. is also itself a sample mean built on integer r.vs. Hence, a general
frame for this kind of problem may be formalized as follows. Let $(q_n)$ be a positive integer sequence. Further, 
let  $\mathbf{X}=(X_j^{(n)})_{n\in\BBn^*, j=1,\ldots,nq_n}$ and 
$\mathbf{Y}=(Y_j^{(n)})_{n\in\BBn^*, j=1,\ldots,nq_n}$ be 
two triangular arrays of random variables. Both arrays are such that on their
lines the r.vs are i.i.d.. Moreover, it is assumed that the elements of the array $\mathbf{X}$ are integer. The interesting distribution is then the law of 
$(nq_n)^{-1}T_{n}:=(nq_n)^{-1}\sum_{j=1}^{nq_n}Y_j^{(n)}$ 
conditioned on a specific value of $S_n:=\sum_{j=1}^{nq_n}X_j^{(n)}$. That is
the conditional distribution
$${\mathcal{L}}_n:={\mathcal{L}}((nq_n)^{-1}T_n|S_n=np_n),$$
where $(p_n)$ is some given positive integer sequence. When the distribution of $(X_j^{(n)},Y_j^{(n)})$ does not depends on $n$, the Gibbs conditioning principle (\cite{VanC,Csis,DeZ98}) states that ${\mathcal{L}}_n$ converges weakly to the degenerated distribution concentrated on a point $\chi$ depending on the conditioning value (see Corollary \ref{pasencore}). Around the Gibbs conditioning principle,     
general limit theorems yielding the asymptotic behavior of the conditioned sum are given in 
\cite{Steck57,Holst79,Kud84}. Asymptotic expansions for the distribution
of the conditioned sum are proved in \cite{Hiphop,RoHHQ90}.
In this paper our aim is to prove a large deviation principle for
$\mathcal{L}_n$. Roughly speaking, this means that we will give an
exponential equivalent for this conditional distribution. On a finer scale, we prove a large deviation principle for
$\tilde{\mathcal{L}}_n:= \mathcal{L}(\sqrt{\frac{a_n}{nq_n}}(T_n-b_n)|S_n=np_n)$, where
$b_n$ is a centering factor specified in Theorem
\ref{modev} of Section \ref{secmod} and $a_n$ is a decreasing positive sequence
of real numbers with $a_n \rightarrow 0$, $na_nq_n \rightarrow + \infty$.
We then say that $\tilde{\mathcal{L}}_n$ satisfies a moderate deviation
principle \cite[Section 3.7]{DeZ98}. 
Our work follows the nice ones of Janson \cite{Ja01,Ja011}. In these last papers, a central limit 
theorem with moment convergence is proved. The starting point in the proof
is a simple representation of the conditional characteristic function as an
inverse Fourier transform. This representation was first given by Bartlett
\cite[Equation (16)]{Bar1938}. To establish large and moderate deviation principles we will make use of G\"artner-Ellis Theorem in which an asymptotic evaluation of the Laplace transform is needed. For this purpose, we first transcribe the Bartlett formula to get a simple integral representation for the conditional Laplace transform 
(see Lemma \ref{lfourfour}). The main result of \cite{Ja01} is quite general
as it only requires assumptions on the three first moments of
$(\mathbf{X},\mathbf{Y})$. Here we need further assumptions. However,
contrarily to \cite[Section 2]{Ja01}, we do not restrict to the central case
(conditioning on $S_n=\mathbb{E} (S_n)$) nor on the ``pseudo'' central case
(conditioning on $S_n=\mathbb{E}(S_n) + \mathcal{O}\left(
\sqrt{nq_n}\sigma_{X^{(n)}_i} \right)$, with $\sigma^2_X=\Var(X)$). In \cite{RoHHQ90},
the authors study general saddle point approximations for multidimensional
discrete empirical means and obtain an approximation formula for conditional
probabilities. We focus here on the exponential part of this formula,
stating a full large deviation principle (see Theorem \ref{thpgd}). Using
some classical tools of convex analysis we give an explicit natural and
elegant form for the rate function. Furthermore, we complement our study by
stating a moderate deviation principle for the conditional law (see Theorem \ref{modev}). As usual, the rate function is quadratic and the scaling factor is
the asymptotic variance, which can be interpreted here as a residual
variance in some linear regression model, generalizing the factor found in
\cite{Ja01}.  The paper is organized as follows. In the next section, to be
self contained, we first recall some classical results on large deviation
principles. Then we state our main results: a large deviation principle and
a moderate deviation principle
for conditioned sums. Section \ref{spro} is devoted to the proofs. In
Section \ref{discution} we apply our main results to some combinatorial
examples. We also discuss possible extensions to more general models. 
%We
%also give a link with the so-called Gibbs conditioning Principle \cite{DeZ98}.
\section{Main results}
\label{smai}
\subsection{Large and moderate deviations}
\label{sunoas}
\subsubsection{Some generalities}
Let us first recall what is a large deviation principle (L.D.P.) (see for example \cite{DeZ98,Holl}). In the whole paper, $(\vite)$
is a decreasing positive sequence of real numbers with $\lim_{n \rightarrow \infty}a_n=0$.

\begin{dfn}
\label{dldp}
We say that a sequence $(R_{n})$ of probability measures on a measurable
Hausdorff space $(U,\br(U))$ satisfies a  LDP with rate function $I$ and speed
$(\vite)$ if:
\begin{itemize}
\item[i)] $I$ is lower semi continuous (lsc), with values in
$\BBr^{+}\cup\{+\infty\}$.
\item[ii)] For any measurable set $A$ of $U$:
$$-I(\Int A)\leq
\liminf_{n\rightarrow\infty}\vite\log R_{n}(A)\leq
\limsup_{n\rightarrow\infty}\vite\log R_{n}(A)\leq
-I({\Clo A}),$$
where $I(A)=\inf_{\xi\in A}I(\xi)$ and $\Int A$ (resp. $\Clo A$) is the
interior (resp. the closure) of $A$.
\end{itemize}
We say that the rate function $I$ is good if its level set
$\{x\in U:\; I(x)\leq a\}$ is  compact for any $a\geq 0$.
More generally, a sequence of $U$-valued random variables
is said to satisfy a
LDP if their distributions satisfy a LDP.
\end{dfn}

To be self-contained, we also  recall some definitions and results
which will be used in the sequel.
(we refer to \cite{DeZ98,Holl} for more on large deviations). 

\paragraph{Laplace and Fenchel-Legendre transforms}
To begin with, 
let $Z$ be a non negative integer random variable and define the span of $Z$ by $m_Z:=\sup\{m\in\N,\ \exists b\in\N,\ \mbox{Supp}(Z)\subset m\N+b\}$.
Let $\varphi_{Z}$ denote the characteristic function of $Z$. When $Z$
is square integrable, $\sigma_{Z}^{2}$ denotes its variance. For
$\tau$ lying in 
$\mbox{dom}\;\psi_{Z}:=\{\tau\in\BBr:\BBe[\exp(\tau Z)]<+\infty\},$
we define $\psi_{Z}(\tau):=\log\BBe[\exp(\tau Z)]$ as the cumulant generating function of
$Z$. Obviously, $\mbox{dom}\;\psi_{Z}$ contains at least $\BBr^{-}$
and $\psi_{Z}$ is analytic in the interior of
$\mbox{dom}\;\psi_{Z}$. We denote by $R_{Z}$ the interior of the range of $\psi_{Z}^{\prime}$.
It is well known that $R_{Z}$ is a subset of the interior of the convex hull of the support
of $Z$. These two subsets of $\BBr$ coincide whenever $\psi_{Z}$ is
essentially smooth (see definition below).
Further, let $\psi^{*}_{Z}$ denote the Fenchel-Legendre transform of
$\psi_{Z}$ \cite[Definition 2.2.2 p. 26]{DeZ98}. For any $\tau^{*}\in R_{Z}$, there exists a unique
$\tau_{\tau^{*}}\in\mbox{dom}\;\psi_{Z}$
such that $\psi_{Z}^{\prime}(\tau_{\tau^{*}})=\tau^{*}$ and we may define
$Z^{*,\tau^{*}}$ as a r.v.
on $\BBn$ having the following distribution
\begin{equation}\label{clementine}
\BBp(Z^{*,\tau^{*}}=k)=\exp[k\tau_{\tau^*}-\psi_{Z}(\tau_{\tau^{*}})]\BBp(Z=k),\;\;
(k\in\BBn).
\end{equation}
It is well known that $\E\left(Z^{*,\tau^{*}}\right)=\tau^*$.
For more details on the relationships between
$\psi_{Z},\psi_{Z}^{*},Z^{*,\tau^{*}}$
we refer to the book \cite{Bar78}.\\
Let now $(Z,W)$ be a random vector of $\BBr^{2}$. We naturally
extend some of the previous notations
to $(Z,W)$. For example, $\psi_{Z,W}$ is the cumulant generating function built on
$(Z,W)$ defined on $\mbox{dom}\;\psi_{Z,W}\subset\BBr^2$ and
$\psi^{*}_{Z,W}$ denotes the
Fenchel-Legendre transform of $(Z,W)$.
\paragraph{Convex functions}
\label{susucon}
Let $f$ be a proper convex function on $\BBr^{k}$. That is $f$ is convex and valued in $\BBr\cup\{+\infty\}$.
We say that $f$ is essentially smooth 
whenever it is differentiable on the non empty interior of dom$f$ and
it is steep. That is, 
for any vector $c$ lying on the boundary of $\mbox{dom} f$ 
$$\lim_{x\rightarrow c, x\in\Int \mbox{dom} f}\|\nabla f(x)\|=+\infty,$$
where $\nabla f(x)$ denotes the gradient of $f$ at point $x$.
\paragraph{G\"artner-Ellis Corollary.}
\begin{cor}\label{gel} [G\"artner-Ellis, \cite[Theorem 2.3.6 c) p.44]{DeZ98}] Let $(Z_n)$ be a sequence of
random variables valued in $\mathbb{R}$, $(a_n)_n$ a decreasing positive
sequence of real numbers with $\lim_{n \rightarrow \infty}a_n=0$. Define 
$\Lambda_n(\theta)=\ln \E e^{\theta Z_n}$.
Assume that
\begin{enumerate}
\item  for all $\theta\in\R$, $a_n \, \Lambda_n(\theta/a_n)\rightarrow \Lambda(\theta) \in
] - \infty , + \infty ]$,
 \item $0$ lies in the interior of 
 $\mbox{dom}(\Lambda(\theta))$ and $\Lambda(\theta)$ is essentially
smooth and lower semi continuous.
\end{enumerate}
Then $(Z_n)$ satisfies a LDP with good rate function $\Lambda^*$ and speed $a_n$.
\end{cor}
\subsection{Main results}
\label{sumainiid}
\subsubsection{The model}
For $n\in\N^*$, let $(X^{(n)},Y^{(n)})$ be a random vector with
$X^{(n)}\in\BBn$. We assume that $m_{X^{(n)}}=1$ and that
$(X^{(n)},Y^{(n)})$ converges in law to $(X,Y)$ where $X$ is a non
essentially constant non negative integer valued r.v. having also
span $1$. Note that it implies that $\psi_X$ is strictly convex and that $X^{(n)}$ is not essentially constant and that $\psi_{X^{(n)}}$ is strictly convex for $n$ large enough.
Further let $\left((X^{(n)}_{i},Y^{(n)}_{i})\right)_{1\leq i\leq n}$ be an
i.i.d. sample
having the same distribution as $(X^{(n)},Y^{(n)})$. 

Let, for $n\in\BBn^{*}$ and $q_n\in\BBn^{*}$, 
$S_{n}=X^{(n)}_{1}+\cdots+X^{(n)}_{nq_n}$ and $T_{n}=Y^{(n)}_{1}+\cdots+Y^{(n)}_{nq_n}$.
In the whole paper $p_n$ will be a sequence of positive integers such that $\P(S_n=np_n)>0$.
\subsubsection{Large deviations}
\begin{thm}\label{thpgd}
Let $p,q,p_n,q_n\in\BBn^{*}$ such that $p_n/q_n \in R_{X^{(n)}} \to p/q \in R_X$. 
Assume that
\begin{enumerate}
\item the function $\psi_{X,Y}$ is essentially smooth, and let $\tau$ be the unique real such
that $\psi'_{X}(\tau)=p/q$,
\item  $\mbox{dom} \psi_{Y}=\mbox{dom} \psi_{Y^{(n)}}=\R$, \label{hypdom}
\item  \label{hyptech}there exists $r>0$ such that
$I_{\tau}:=[\tau-r,\tau+r]\subset \left( \mbox{dom}\;\psi_{X} \right) \cap
  \left( \cap_{n \geq 1}\mbox{dom}\;\psi_{X^{(n)}}\right)$ and
\begin{equation}\label{hyp1}
\forall \, u \in \R\, , \forall \, s \in I_{\tau} \, , \; \sup_{t \in \R}\left|\E \left[ e^{    (it+s)X^{(n)}+uY^{(n)}   }-e^{
(it+s)X+uY    }\right] \right|\to 0 \, .
\end{equation}
\end{enumerate}
Then the distribution of $(T_n/nq_n)$ conditioned by the event $\{S_{n}=np_n\}$ satisfies a LDP
with good rate function $ \psi^{*}_{X,Y}(p/q,\cdot)-\psi^{*}_{X}(p/q)$ and
speed $(nq_n)^{-1}$.
\end{thm}

\subsubsection{Moderate deviations}\label{secmod}
Let $\xi$ lying in the interior of $\mbox{dom} \psi_{X}$ and consider the random
vector $(\check{X}_{\xi},\check{Y}_{\xi})$ whose distribution is given, for
any $k\in\BBn$ and real Borel set $A$, by
\begin{equation}\label{loicheck}\BBp(\check{X}_{\xi}=k,\check{Y}_{\xi}\in A)=
\exp[-\psi_{X}(\xi)+k\xi] \BBp(X=k,Y\in A).
\end{equation}
We define in the same way the random
vector $(\check{X}^{(n)}_{\xi},\check{Y}^{(n)}_{\xi})$. 
Obviously, $\check{X}_{\xi}$ has the same distribution as $X^{*,\xi^{*}}$
with $\xi^{*}=\psi_{X}'(\xi)$. Further, let $\alpha^{2}_{\xi}$ be the variance of the
residual $\check{\varepsilon}_{\xi}$
for the linear regression of $\check{Y}_{\xi}$ on $\check{X}_{\xi}$:
\begin{equation}\label{alpha2}\check{\varepsilon}_{\xi}:=\Blp\check{Y}_{\xi}-\E\blp \check{Y}_{\xi}\brp\Brp-
\frac{\mbox{cov}(\check{X}_{\xi},\check{Y}_{\xi})}{\mbox{var}(\check{X}_{\xi})}
\Blp\check{X}_{\xi}-\E\blp \check{X}_{\xi}\brp\Brp.
\end{equation}
Then we get the following result.

\begin{thm}\label{modev}
Let
$p,q,p_n,q_n\in\BBn^{*}$ such that $p/q$ (resp. $p_n/q_n$) lies in $R_X$  (resp. $R_{X^{(n)}}$) and $p_n/q_n\to p/q$. Assume that
\begin{enumerate}
\item there exists $r_0>0$ such that
$$B_0:=]-r_0,r_0[ \subset \left(
    \mbox{dom} \psi_{Y} \right) \cap \left( \cap_{n \geq 1} \mbox{dom}
    \psi_{Y^{(n)}}\right) \, ,\label{hypdom2}$$
and let $\tau$ (resp. $\tau_n$) be the unique real such that
$\psi'_{X}(\tau)=p/q$ (resp. $\psi'_{X^{(n)}}(\tau_n)=p_n/q_n$),
\item \label{hyptech2}there exists $r>0$ such that
$I_{\tau}:=[\tau-r,\tau+r] \subset \left(\mbox{dom}\;\psi_{X}\right) \cap
  \left( \cap_{n \geq 1}\mbox{dom}\;\psi_{X^{(n)}}\right)$ and  
\begin{align}
\forall \, s \in I_{\tau} \, , \; \sup_{t \in \R} &\left|\E \left[ e^{    (it+s)X^{(n)}   }-e^{(it+s)X    }\right] \right|\to 0 \, ,\label{hyp2}\\
\sup_n\sup_{(s,v)\in I_\tau\times B_0} &\E\left(
e^{sX^{(n)}+v(Y^{(n)}-\EMD)}\right)<\infty \, ,\label{hyp3}
\end{align}

\item $(a_n)$ satisfies $na_nq_n\to+\infty$. 
\end{enumerate}
Then the distribution of $\Blp \sqrt{\frac{a_n}{nq_n}}\blp T_n-nq_n\EMD\brp\Brp$ conditioned by the event $\{S_{n}=np_n\}$ satisfies a LDP
with good rate function $J(\cdot)= \frac{(\cdot)^2}{2\alpha^{2}_{\tau}}$ and
speed $a_n$. 
\end{thm}
\begin{rem}
As $(nq_n)^{-1}=o(a_n)$, we say that the distribution of $\Blp
\sqrt{\frac{a_n}{nq_n}}\blp T_n-nq_n\EMD\brp\Brp$ conditioned by the event
$\{S_{n}=np_n\}$ satisfies a moderate deviation principle (MDP).
\end{rem}
\begin{cor} 
\label{pasencore} 
Under the assumptions of one of the last theorems, ${\mathcal{L}}_n$ converges in distribution toward's the degenerate distribution concentrated on 
$E\blp \check{Y}_{p/q}\brp$. 
\end{cor} 
\section{Proofs}
\label{spro}
For $p_n \in \N$, such that $\P(S_n=np_n)\neq 0$, let
$$f_{n}(u):=\frac1{nq_n}\log\BBe\left[\exp(u T_{n})\left|S_{n}=np_n\right.\right]
\in\BBr\cup\{+\infty\}.$$
In order to apply G\"artner-Ellis Corollary, we have to prove that $f_n(u)$
converges when $n\to\infty$. The next two subsections yield a simple representation of $f_n(u)$ using the Fourier Transform.
\subsection{A simple representation using Fourier Transform}
Recall that we set $\varphi_Z(t):=\E\left(e^{itZ}\right)$.
\label{sufour}
An obvious but useful lemma follows.
\begin{lem}[Bartlett's Formula, see Equation (16) in \cite{Bar1938}]
\label{lfourfour}
Let $Z$ be a non negative integer r.v. and $W$ be an integrable r.v. Then, for
any non negative integer $k$ lying in the support of $Z$,
$$\BBe\left[W|Z=k\right]=\frac{\int_{-\pi}^{\pi}\BBe[W\exp(itZ)]\exp(-ikt)dt}
{\int_{-\pi}^{\pi}\varphi_{Z}(t)\exp(-ikt)dt} \, .
$$  
\end{lem}
\subsection{Laplace lemmas}
\label{sulapiid}
We begin this section with a variation on a Lemma first due to Laplace, see \cite{Hald98} . 
To be self
contained we give also the sketch of its proof.
\begin{lem}\label{laplace1}
Let $p_n,q_n,p,q\in\N^*$. Let $Z$ be a non constant square integrable non negative integer r.v. with span $m_Z=1$. Let
$(Z^{(n)})$ be a sequence  of non negative i.i.d. integer random variables also having  span $1$.  
Let $Z_{1}^{(n)},\ldots,Z_{n}^{(n)}$ be an i.i.d. sample
distributed as $Z^{(n)}$. Assume that
\begin{enumerate}
\item  $\|\varphi_{Z^{(n)}}-\varphi_{Z}\|_{\infty} \xrightarrow[n
  \rightarrow + \infty]{}0$,\label{hyplem1}
\item \label{hyplem2}the means of $Z$ and $Z^{(n)}$ are  rational,
equal respectively  to $p/q$ and $p_n/q_n$ with $p_n/q_n\xrightarrow[n
  \rightarrow + \infty]{} p/q$,
\item  \label{hyplem3}$\sigma_{Z^{(n)}}^2 \xrightarrow[n \rightarrow + \infty]{} \sigma_Z^2$,
\item \label{hyplem4}
$\E\blp|Z^{(n)}-p_n/q_n|^3\brp$ is uniformly bounded.
\end{enumerate}
Then, when $n$ tends to infinity
\begin{equation}
\label{lala}
\BBp\left(\sum_{j=1}^{nq_n}Z_{j}^{(n)}=np_n\right)=\frac{1}{\sqrt{2\pi nq_n}\sigma_{Z}}(1+o(1)).
\end{equation}
\end{lem}
\begin{proof}
The inversion of the Fourier Transform yields
\begin{equation}
\BBp\left(\sum_{j=1}^{nq_n}Z_{j}^{(n)}=np_n\right)=\ipp
e^{-inp_nt}\varphi_{Z^{(n)}}^{nq_n}(t)\frac{dt}{2\pi}
=\ipp
\left[e^{-i\frac{p_n}{q_n}t}\varphi_{Z^{(n)}}(t)\right]^{nq_n}\frac{dt}{2\pi}.
\label{foufou}
\end{equation}
On one hand, using a Taylor expansion of order $2$ for
$e^{-i\frac{p_n}{q_n}t}\varphi_{Z^{(n)}}(t)=\varphi_{Z^{(n)}-\frac{p_n}{q_n}}(t)$,
we get
$$
\varphi_{Z^{(n)}-\frac{p_n}{q_n}}(t)=1-\frac{t^2}2\sigma^2_n-\frac{t^3}6\E\blp i(Z^{(n)}-p_n/q_n)^3e^{it^*Z^{(n)}}\brp,
$$
where $t^*$ lies in $[0,t]$ and $\sigma^2_n$ states for $\sigma^2_{Z^{(n)}}$.
Now as $\E|(Z^{(n)}-p_n/q_n)|^3$ is bounded and as $\sigma^2_n\to\sigma^2_Z$ we can find  a
positive number  (independent of $n$) $\delta<\pi$ such that for $|t|<\delta$ and for $n$ large enough 
\begin{equation}
\left|e^{-i\frac{p_n}{q_n}t}\varphi_{Z^{(n)}}(t)\right|=\left|\varphi_{Z^{(n)}-\frac{p_n}{q_n}}(t)\right|
\leq 1-\frac{\sigma_Z^{2}t^{2}}{4}.
\label{tata}
\end{equation}
On the other hand, as $m_{Z^{(n)}}=m_Z=1$  one has both 
$\xi:=\sup_{\delta\leq t\leq \pi}|\varphi_Z(t)|<1$ 
and
$\xi_n:=\sup_{\delta\leq t\leq \pi}|\varphi_{Z^{(n)}}(t)|<1$. 
Further, as
$\|\varphi_{Z^{(n)}}-\varphi_{Z}\|_{\infty}\to0$, we get $\xi_n\to\xi$. Let
$\epsilon>0$ be such that $\xi+\epsilon<1$. For $n$ large enough $\xi_n\leq \xi+\epsilon$. Now to conclude,
one splits the integral in (\ref{foufou}) in two integrals $I_{1}$, $I_{2}$ integrating
on $|t|<\delta$  and on $|t|\geq\delta$. $|I_{2}|$ is bounded by
$\blp\xi+\epsilon\brp^{nq_n}/(2 \pi)$, hence is exponentially small. To deal with $I_{1}$, one
performs the variable change $u=\sqrt{nq_n}\sigma_{Z}t$, and use both
(\ref{tata}) and inequality $\log(1-\theta)\leq-\theta, (\theta\in[0,1[)$  to conclude
 by using both central limit and Lebesgue Theorems.\QED
\end{proof}
\begin{rem}
Note that if $m_Z\not=1$, the theorem above is in general not valid. Take for example $m_Z=2$ then in the interval $[0,2\pi[,\ $ $\varphi_{Z^{(n)}}(t)=\varphi_Z(t)=1$ if and only if $t=0$ or $t=\pi$.
Obviously,
$$\usdp\ipp
e^{-inp_nt}\varphi_{Z^{(n)}}^{nq_n}(t)dt=\usdp\int_{-\pi/2}^{3\pi/2}
e^{-inp_nt}\varphi_{Z^{(n)}}^{nq_n}(t)dt . $$
Hence
\begin{equation}\label{foufou2}
\BBp\left(\sum_{j=1}^{nq_n}Z_{j}^{(n)}=np_n\right)=\usdp\int_{-\pi/2}^{3\pi/2}
e^{-inp_nt}\varphi_{Z^{(n)}}^{nq_n}(t)dt.
\end{equation}
Set $f_n(t):=e^{-inp_nt}\varphi_{Z^{(n)}}^{nq_n}(t)dt$.
Now, we split the integral in the right hand side of equation (\ref{foufou2}) in five parts
\begin{align*}
\usdp\int_{-\pi/2}^{3\pi/2}f_n(t)dt=&
\usdp\int_{-\pi/2}^{-\delta}f_n(t)dt+
\usdp\int_{-\delta}^{\delta}f_n(t)dt+
\usdp\int_{\delta}^{\pi-\delta}f_n(t)dt\\
&+\usdp\int_{\pi-\delta}^{\pi+\delta}f_n(t)dt+
\usdp\int_{\pi+\delta}^{3\pi/2}f_n(t)dt.\\
&=I_1+I_2+I_3+I_4+I_5.
\end{align*}
Using the same arguments as above, we can prove that $I_1,\ I_3$ and $I_5$ are exponentially small. We also get that $\usdp\int_{-\delta}^{\delta}f_n(t)dt=\frac1{\sqrt{2\pi nq_n}\sigma_Z}(1+o(1))$. Let us deal now with $I_4=\usdp\int_{\pi-\delta}^{\pi+\delta}f_n(t)dt$. There exist   non negative integer valued  random variables $Y_n$ such that $Z^{(n)}=2Y_n+b$, set $u=t-\pi$
$$
I_4=\frac1{2\pi}(-1)^{np_n+bnq_n}\int_{-\delta}^\delta e^{-inp_nu}\varphi_{Y_n}^{nq_n}(2u)du.
$$ 
Hence $I_4\not=\frac1{\sqrt{2\pi nq_n}\sigma_Z}(1+o(1))$.
\end{rem}

We now give an extension of the previous lemma involving not only the
probability for the sum to be equal to the mean of $Z^{(n)}$ but to any {\it
  good} rational number.
\begin{lem}\label{lempertub}
Let $Z$ be a non negative and non degenerated integer r.v. with span $m_Z=1$. Let
$(Z_j^{(n)})_j$ be a sequence  of i.i.d. non negative integer random
variables having also span $1$. Let $p_n,q_n,p,q\in\BBn^{*}$ such that
$p_n/q_n\in R_{Z^{(n)}}\to p/q\in R_{Z}$.
Let $\tau$ (resp. $\tau_n$) be the unique real such that
  $\psi_Z'(\tau)=p/q$ (resp. $\psi_{Z^{(n)}}'(\tau_n)=p_n/q_n$).
We make the following assumptions.
\begin{enumerate}
\item  There exists $r>0$ such that $I_{\tau}:=[\tau-r,\tau+r] \subset
  \left(\cap_{n \geq 1}\mbox{dom}\;\psi_{Z^{(n)}}\right) \cap \left(\mbox{dom}\;\psi_{Z}\right)$.   
\item 
\begin{equation}\label{cvunif}
\forall \, s \in I_{\tau} \, , \; \lim_n \, \sup_{t\in \R} \left| \E \left[ \blp e^{(it+s)
Z^{(n)}}-e^{(it+s) Z}\brp\right]\right|=0 \, .
\end{equation}
\end{enumerate}
Then, when $n$ goes to
infinity
\begin{equation}
\label{la1la1}
\BBp\left(\sum_{j=1}^{nq_n}Z_{j}^{(n)}=np_n\right)=e^{-nq_n\psi^{*}_{Z^{(n)}}(p_n/q_n)}\frac{1}{\sqrt{2\pi nq_n}\sigma_{Z^{*,p/q}}}(1+o(1)),
\end{equation}
\label{llap1}
where $\sigma^2_{Z^{*,p/q}}$ is the variance of $Z^{*,p/q}$ defined in (\ref{clementine}).
\end{lem}
\begin{proof}
Using the multinomial formula, we may write
$$
\BBp\left(\sum_{j=1}^{nq_n}Z_{j}^{(n),*,p_n/q_n}=np_n\right)=e^{ nq_n(\frac{p_n}{q_n}\tau_n-\psi_{Z^{(n)}}(\tau_n))}\BBp\left(\sum_{j=1}^{nq_n}Z^{(n)}_{j}=np_n\right),
$$
where $\tau_n$ is the unique real such that $\psi'_{Z^{(\tau_n)}}=p_n/q_n$.
Hence
$$\BBp\left(\sum_{j=1}^{nq_n}Z^{(n)}_{j}=np_n\right)=e^{-nq_n\psi^{*}_{Z^{(n)}}(p_n/q_n)}\BBp\left(\sum_{j=1}^{nq_n}Z_{j}^{(n),*,p_n/q_n}=np_n\right),$$
where $Z_{1}^{(n),*,p_n/q_n},\ldots, Z_{n}^{(n),*,p/q}$ are i.i.d. r.vs.
having the distribution defined by 
$$
\P\blp  Z^{(n),*,p_n/q_n}=k\brp=\exp\blp k\tau_n-\psi_{Z_{1}^{(n)}}(\tau_n)\brp\P\blp Z_{1}^{(n)}=k\brp.
$$
Further, the expectation of this last r.v. is $p_n/q_n$. Let us now check the assumptions of Lemma \ref{laplace1}.
\begin{itemize}
\item Assumption \ref{hyplem2} of Lemma \ref{laplace1} is satisfied by construction of $Z^{(n),*,p_n/q_n}$.
\item Let us  prove that $ 
\E\left|Z^{(n),*,p_n/q_n}-p_n/q_n\right|^3
$
is bounded.
Using H\"older inequality we get that
\begin{align*} 
\E\left|Z^{(n),*,p_n/q_n}-p_n/q_n\right|^3&=\sum_{k=0}^\infty\left|k-\frac{p_n}{q_n}\right|^3e^{k\tau_n-\psi_{Z^{(n)}}(\tau_n)}\P\blp Z_{1}^{(n)}=k\brp\\
&=e^{\psi^*_{Z^{(n)}}(p_n/q_n)}\E\left(\left|Z^{(n)}-\frac{p_n}{q_n}\right|^3e^{\tau_n(Z^{(n)}-p_n/q_n)}
\right)\\
&\leq e^{\psi^*_{Z^{(n)}}(\frac{p_n}{q_n})}\left(\E\left((Z^{(n)}-\frac{p_n}{q_n})^4e^{\tau_n(Z^{(n)}-\frac{p_n}{q_n})}
\right)\right)^{3/4}\left(\E\left(e^{\tau_n(Z^{(n)}-\frac{p_n}{q_n})}
\right)\right)^{1/4}.
\end{align*}
\begin{comment}
Let us prove that $\tau_n$ converges to $\tau$. By definition,
$\tau_n=(\psi_{Z^{(n)}}^*)'(p_n/q_n)$ and  $\tau=(\psi_Z^*)'(p/q)$. Let
$\tau_n^{*}:=(\psi_{X_n}^*)'(p/q)$. We have
\begin{equation}\label{bidule}
\tau_n-\tau=(\psi_{Z^{(n)}}^*)'(p_n/q_n)-(\psi_{Z^{(n)}}^*)'(p/q)+(\psi_{Z^{(n)}}^*)'(p/q)-(\psi_{Z}^*)'(p/q).
\end{equation}
 Now, as the convergence of $\psi_{Z^{(n)}}$  to $\psi_Z$ implies the
 convergence of $\psi_{Z^{(n)}}^*$  to $\psi_Z^*$, we can apply Theorem 25.7
 page 248 in Rockafellar \cite{Roc72} to the second part of the right hand
 side of (\ref{bidule}). Hence $(\psi_{Z^{(n)}}^*)'(p/q)-(\psi_{Z}^*)'(p/q)\to0$. Let us deal now with $(\psi_{Z^{(n)}}^*)'(p_n/q_n)-(\psi_{Z^{(n)}}^*)'(p/q)$. There exists a point $\xi_n\in(p_n/q_n,p/q)$ such that
$$(\psi_{Z^{(n)}}^*)'(p_n/q_n)-(\psi_{Z^{(n)}}^*)'(p/q)=(\frac{p_n}{q_n}-\frac{p}q)(\psi_{Z^{(n)}}^*)''(\xi_n).$$
We will show that $(\psi_{Z^{(n)}}^*)^{''}$ is bounded. Now
$(\psi_{Z^{(n)}}^*)^{''}(\xi)=\frac1{\psi_{Z^{(n)}}^{''}(\psi^{'-1}_{Z^{(n)}}(\xi))}$.
By the hypotheses  $\psi_{Z^{(n)}}^{''}$ converges. Moreover, for $n$ large
enough $\xi_n\in[p/q-\epsilon,p/q+\epsilon]$. It concludes the proof of
$\tau_n\to\tau$. 
\end{comment}
Using classical arguments on convex functions \cite{Roc72}, we get that $\tau_n
\xrightarrow[n \rightarrow + \infty]{} \tau$.
Hence, by Assumptions 1. and 2. of Lemma \ref{lempertub}
\begin{comment}
as there exists a neighborhood of $\tau$ contained
 in $\cap_{n \geq 1}  \mbox{dom}\;\psi_{Z^{(n)}} $,
 (\ref{cvunif}) implies that 
\end{comment}
we get that
$ 
\E\left|Z^{(n),*,p/q}-p/q\right|^3
$
is bounded.
\item  Similar arguments yield that
\begin{equation}\label{cvunif2}
\|\varphi_{Z_1^{(n),*,p_n/q_n}}-\varphi_{Z^{*,p/q}}\|\to0
\end{equation}
and that
\begin{comment}
We may write,
$$
\varphi_{Z^{(n),*,p/q}}(t)=\frac{\E\blp e^{(\tau_n+it)Z^{(n)}}\brp}{\E\blp e^{(\tau_nZ^{(n)}}\brp}.
$$
Hence, as $\tau_n\to\tau$, and as there exists a neighborhood of $\tau$ contained
 in $\left(\cap_{n \geq 1}  \mbox{dom}\;\psi_{Z^{(n)}}
 \right)$, (\ref{cvunif}) implies
(\ref{cvunif2}).
\item Now, deriving twice (\ref{cvunif}) we get that 
\end{comment}
$\sigma_{Z_{n}^*}^2 \rightarrow \sigma_{Z^*}^2$.
\end{itemize}
Hence all the assumptions of Lemma \ref{laplace1} are satisfied and we may
conclude using Lemma \ref{laplace1}.\QED
\end{proof}
\subsection{Some changes of probability}\label{newprob}
One of the main tool to prove large deviation results is the use of changes of
probability. In this section, we review the different changes of probability
used in this paper.
\begin{enumerate}
\item[(a)] \label{change1} Let $p/q\in R_{Z}$. We define
  $\tau\in\mbox{dom}\;\psi_{Z}$ by $\psi_{Z}^{\prime}(\tau)=p/q$. We then
  introduce $Z^{*,p/q}$ as a random variable valued 
on $\BBn$:
\begin{equation}
\BBp(Z^{*,p/q}=k)=\exp[k\tau-\psi_{Z}(\tau)]\BBp(Z=k),\;\;
(k\in\BBn).
\end{equation}
We have $\mathbb{E}(Z^{*,p/q})=\frac{p}{q}$. This change of
probability is quite classical in large deviation theory. In order to prove
Lemma \ref{lempertub}, we also define $Z^{(n),*,p_n/q_n}$, replacing $Z$ by
$Z^{(n)}$, $p/q$ by $p_n/q_n$ and $\tau$ by $\tau_n$. Then $\E\blp Z^{(n),*,p_n/q_n}\brp=p_n/q_n$, as needed to apply Lemma \ref{laplace1}. 
\item[(b)] \label{change2}
For $u$ in $\mbox{dom}\psi_Y$, define
 $\widehat{X}_{u}$ by
\begin{equation}\label{newchap}
\BBp\left(\widehat{X_{u}}=k\right)=
\exp\left[-\psi_{Y}\blp u\brp\right]\BBe\left[\exp\blp u Y\brp\BBone_{\{X=k\}}\right].
\end{equation}
Similarly, replacing $(X,Y)$ by $\left(X^{(n)},Y^{(n)}\right)$, we define $\widehat{X}_u^{(n)}$. The r.v. $\widehat{X_{u}}$ and $\widehat{X}_u^{(n)}$ appear naturally when applying the inversion of Fourier transform in the proof of Theorem \ref{thpgd}.
\item[(c)] \label{change3}For the moderate deviations, the asymptotic is
  different (see Theorem \ref{modev}). Therefore the r.v.
  $Y^{(n)}$ have to be centered. The centering factor and the rate function
  are closely related to the following change of probability.
Let $\xi$ lying in the interior of $\mbox{dom} \psi_{X}$ and consider the random
vector $(\check{X}_{\xi},\check{Y}_{\xi})$ whose distribution is given, for
any $k\in\BBn$ and real Borel set $A$, by
\begin{equation}\BBp(\check{X}_{\xi}=k,\check{Y}_{\xi}\in A)=
\exp[-\psi_{X}(\xi)+k\xi] \BBp(X=k,Y\in A).
\end{equation}
We define in the same way the random
vector $(\check{X}^{(n)}_{\xi},\check{Y}^{(n)}_{\xi})$. 
Obviously, $\check{X}_{\xi}$ has the same distribution as $X^{*,\xi^{*}}$
with $\xi^{*}=\psi_{X}'(\xi)$. Further, let $\alpha^{2}_{\xi}$ be the variance of the
residual $\check{\varepsilon}_{\xi}$
for the linear regression of $\check{Y}_{\xi}$ on $\check{X}_{\xi}$:
\begin{equation}\check{\varepsilon}_{\xi}:=\Blp\check{Y}_{\xi}-\E\blp \check{Y}_{\xi}\brp\Brp-
\frac{\mbox{cov}(\check{X}_{\xi},\check{Y}_{\xi})}{\mbox{var}(\check{X}_{\xi})}
\Blp\check{X}_{\xi}-\E\blp \check{X}_{\xi}\brp\Brp.
\end{equation}
Note that $J(y)=\frac{y^2}{\alpha^2_\tau}$ is the rate function in Theorem
\ref{modev}. Moreover, the centering factor is $\EMD$. Hence, the change of
probability used in the proof of Theorem \ref{thpgd} (see change of
probability (\ref{newchap}) above) has to be modified, according to this
centering factor. This leads to the change of probability (\ref{horreur}) below. 
\item[(d)] \label{change4}
Let $\tau$ (resp. $\tau_n$) be such that $\psi_{X}'(\tau)=p/q$ (resp. $\psi_{X^{(n)}}'(\tau_n)=p_n/q_n$). Define 
 the random variable $\Hor$ distributed on $\N$ by:

\begin{equation}\label{horreur}\BBp\left(\Hor=k\right)=
e^{-\psi_{\widetilde{Y}^{(n)}}( u_n)}\BBe\left[e^{\frac{u \widetilde{Y}^{(n)}}{\sqrt{na_nq_n}}}\BBone_{\{X^{(n)}=k\}}\right],\end{equation}
where $\widetilde{Y}^{(n)}=Y^{(n)}-\EMD$, $u_n=u/\sqrt{na_nq_n}$. 
\end{enumerate}
\subsection{Proof of Theorem \ref{thpgd}}
Let, for $t\in\BBr$ and $u\in\R$,
$$\Phi_{X^{(n)},Y^{(n)}}(t,u):=\BBe\left(\exp[itX^{(n)}+u Y^{(n)}]\right).$$
On one hand, using Lemma \ref{lfourfour}, we may write, for $u\in\BBr$ 
and $n$ large enough,
\begin{equation}
f_{n}(u)=\frac{1}{nq_n}\log
\frac{\int_{-\pi}^{\pi}e^{-inp_nt}\Phi_{X^{(n)},Y^{(n)}}^{nq_n}(t,u)dt}
{\int_{-\pi}^{\pi}e^{-inp_nt}\Phi_{X^{(n)},Y^{(n)}}^{nq_n}(t,0)dt}.
\label{llave}
\end{equation}
 Using twice equation (\ref{foufou}) we may rewrite (\ref{llave}) as 
\begin{equation}
f_{n}(u)=\frac{1}{nq_n}\left[\log\BBp\left(\sum_{j=1}^{nq_n}\widehat{X}^{(n)}_{u,j}
=np_n\right)-\log\BBp\left(S_{n}=np_n\right)\right]+\psi_{Y^{(n)}}(u),
\label{llave1}
\end{equation}
where $\widehat{X}^{(n)}_{u,1},\cdots,\widehat{X}^{(n)}_{u,n}$ are independent copies of $\widehat{X}^{(n)}_u$ defined  in Subsection \ref{newprob} by equation (\ref{newchap}). 
In order to apply  Lemma \ref{llap1} to $\widehat{X}^{(n)}_{u,i}$, let us
prove that
\begin{equation}\label{cvunif3}
\forall \, s \in I_{\tau} \; \lim_n\sup_{t\in \R} \left| \E\left[\blp e^{(it+s)\widehat{X}^{(n)}_{u}} -e^{(it+s) \widehat{X}_{u}}\brp\right]\right|=0.
\end{equation}
We have, for all $s \in I_{\tau}$,
\begin{align*}
\left| \E\left[\blp e^{(it+s)\widehat{X}^{(n)}_{u}} -e^{(it+s)
\widehat{X}_{u}}\brp\right]\right|&\leq Ce^{-\psi_Y(u)}\left|\E\left[ e^{u Y^{(n)}} e^{(s+it) X^{(n)}}- e^{uY}e^{(s+it) X }\right]\right|.
\end{align*}
The right hand side of this last inequality tends to $0$ by assumption 3. of Theorem \ref{thpgd}.
It remains to prove that $p_n/q_n$ (resp. $p/q$) belongs to
$R_{\hat{X}_u^{(n)}}$ (resp. $R_{\hat{X}_u}$). Using the fact that $\psi_{X,Y}$ is essentially smooth and Assumption 3. of Theorem \ref{thpgd} it is easy to see that $R_{\hat{X}_u}=R_X$ (resp. $R_{\hat{X}_u^{(n)}}=R_{X^{(n)}}$ at least for $n$ large enough). 

Applying  Lemma \ref{llap1}
 we obtain, for $u\in\BBr$
$$f(u):=
\lim_{n\rightarrow\infty}
\frac{1}{nq_n}\log\BBe\left[\exp(u T_{n})\left|S_{n}=np_n\right.\right]
=-[\psi^{*}_{\widehat{X_{u}}}(p/q)-\psi_{Y}(u)-\psi^{*}_{X}(p/q)].$$
The convex dual function $f^{*}$ of $f$ is given by
\begin{eqnarray}
f^{*}(y):=\sup_{u\in\BBr}[u y-f(u)]&=&
\sup_{u\in\BBr}\left(u y+
\left[\psi^{*}_{\widehat{X_{u}}}(p/q)-\psi_{Y}(u)\right]\right)
-\psi^{*}_{X}(p/q)\nonumber
\\
&=&\sup_{(u,\xi)\in{\mbox{dom} \psi_{X,Y}}}\left[
uy+\xi\frac{p}{q}-\psi_{X,Y}(\xi,u)\right]-\psi^{*}_{X}(p/q)\nonumber
\\
&=& \psi_{X,Y}^{*}\left(\frac{p}{q},y\right)-\psi^{*}_{X}(p/q).\label{dudu}
\end{eqnarray}
As $\psi_{X,Y}$ is essentially
smooth, using Theorem 26.3 in \cite{Roc72}, we deduce that $\psi_{X,Y}^{*}$ is essentially strictly convex. Hence, using
once more Theorem 26.3 in \cite{Roc72}, we may deduce that $f$ is essentially smooth.
Therefore we can apply G\"artner-Ellis Corollary \ref{gel} (see Theorem
2.3.6. (c) in
\cite{DeZ98}) and conclude.\QED
\subsection{Proof of Theorem \ref{modev}}
Let $\tilde{T_n}=T_n-nq_n\EMD$ and
$$
g_n(u)=a_n\log\Blp\E\blp e^{\tilde{T}_nu / \sqrt{na_nq_n}}|S_n=np_n\brp \Brp.
$$
Proceeding as in the proof of Theorem \ref{thpgd}, we have
\begin{align*}
g_n(u)&=a_n\log\frac{\int_{-\pi}^\pi e^{-inp_nt}\Phi^{nq_n}_{X^{(n)},Y^{(n)}-\EMD}(t,u/\sqrt{na_nq_n})}{ \int_{-\pi}^\pi e^{-inp_nt}\Phi^{nq_n}_{X^{(n)},Y^{(n)}-\EMD}(t,0)}\\
&= a_n\Blp\log\P\blp\sum_{j=1}^{nq_n}\Horj=np_n\brp-\log\P\blp S_n=np_n\brp\Brp+a_nnq_n\psi_{Y^{(n)}-\EMD}(u/\sqrt{na_nq_n}),
\end{align*}
where $\Horj$ are i.i.d. r.v. on $\N$ with distribution defined in Subsection \ref{newprob} by  Equation (\ref{horreur}).

In order to use Lemma \ref{llap1} we first have to prove that
\begin{equation}\label{cvunif6}
\forall \, s \in I_{\tau} \, , \; \lim_n\sup_{t\in \R} \left| \E\left[e^{(it+s)\Hor} -e^{(it+s) \Horsn}\right]\right|=0.
\end{equation}
We have
$$
\left| \E \left[ e^{(it+s) \Hor} -e^{(it+s)
\Horsn }\right] \right|  \leq C e^{-\psi_{Y-\EMDA}(0)} \left| \E \left[
e^{(s+it) X^{(n)}}- e^{(s+it) X }\right]\right|,
$$
which tends to zero by assumption (\ref{hyp2}). As in the proof of Theorem \ref{thpgd} it is easy to prove that $R_{\Horsn}=R_X$ and $R_{\Hor}=R_{X^{(n)}}$.

Using Lemma \ref{llap1} we obtain, for $u\in\BBr$
$$g_n(u)\stackrel{n\to+\infty}{\sim}
-q_nna_n\left[\psi^{*}_{\Hor}(p_n/q_n)-\psi_{Y^{(n)}-\EMD}(u/\sqrt{na_nq_n})-\psi^{*}_{X^{(n)}}(p_n/q_n)\right].$$
Define
$$H_n(h)=\sup_{\xi\in\BBr}\left[\xi \frac{p_n}{q_n}-\psi_{X^{(n)},Y^{(n)}-\EMD}(\xi,h)\right].$$
As
$$
\psi^{*}_{\Hor}(p_n/q_n)=\sup_x\left(\frac{p_n}{q_n}x-\psi_{\widehat{X}^{(n)}_{u}}(x)\right),$$
and
$$\psi_{\widehat{X}^{(n)}_{u}}(x)=\psi_{X^{(n)},Y^{(n)}-\EMD}(x,u/\sqrt{na_nq_n})-\psi_{Y^{(n)}-\EMD}(u/\sqrt{na_nq_n})\, ,$$
we get
\begin{equation}\label{equivg}
g_n(u)\stackrel{n\to+\infty}{\sim}-q_nna_n\blp
H_n(u/\sqrt{na_nq_n})-H_n(0)\brp\, .
\end{equation}
We claim that if $\lim_nh_n=0$, then 
\begin{equation}\label{hn}
\lim_n\frac{H_n(h_n)-H_n(0)}{h_n^2}=-\frac{\alpha^2_{\tau_n}}2+O(1).
\end{equation}
Assuming that (\ref{hn}) is true, and as $\alpha^2_{\tau_n}\to \alpha^2_\tau$, we get that

$$
\lim_ng_n(u)=-u^2\frac{\alpha^2_{\tau}}{2}.
$$
We easily conclude, since
$$
g^*(y)=\sup_u\left\{uy+\lim_n g_n(u)\right\}=\frac{y^2}{2\alpha^2_{\tau}}.
$$
It remains to prove that (\ref{hn}) is true. Recall that 
$$H_n(h)=\sup_{\xi\in\BBr}\left[\xi \frac{p_n}{q_n}-\psi_{X^{(n)},Y^{(n)}-\EMD}(\xi,h)\right].$$
In the sequel $\psi'_x$ (resp. $\psi'_y$) will denote the partial derivative
of $\psi_{X^{(n)},Y^{(n)}-\EMD}(\xi,h)$ with respect to the first
(resp. second) variable. 
On one hand, by assumption (\ref{hyp3}), we can define on $I_{\tau} \times B_0$ the function $F_n$ by:
$$F_n(\xi,h)=\psi'_x(\xi,h)-p_n/q_n \, .$$ We then deduce from the implicit function
Theorem that there exists a neighborhood of $(\tau_n,0)$ on which:
$$H_n(h)=\xi_n(h) \frac{p_n}{q_n}-\psi_{X^{(n)},Y^{(n)}-\EMD}(\xi_n(h),h) \,
,$$
with $$\xi_n'(h)=-\frac{\psi''_{xy}(\xi_n(h),h)}{\psi''_{xx}(\xi_n(h),h)} \, .
$$
We can then calculate the derivatives of $H_n$ (in the sequel we omit the
argument $(\xi_n(h),h)$ in the derivatives). We have (with obvious notations) 
\begin{align*}
H'_n(h)&=-\psi'_y,\\
H''_n(h)&=\frac{\left(\psi''_{x,y}\right)^2}{\psi''_{x,x}}-\psi''_{y,y},\\
H^{(3)}_n(h)&=\left(\frac{\psi''_{x,y}}{\psi''_{x,x}}\right)^3\psi^{(3)}_{x,x,x}
-3\left(\frac{\psi''_{x,y}}{\psi''_{x,x}}\right)^2\psi^{(3)}_{x,x,y}
+3\frac{\psi''_{x,y}}{\psi''_{x,x}}\psi^{(3)}_{x,y,y}
-\psi^{(3)}_{y,y,y} \, .
\end{align*}
Replacing the partial derivative of $\psi$ by its expression, we get 
$$
H'_n(0)=0\mbox{ and } H''_n(0)=-\alpha^2_{\tau_n} \, .
$$
On the other hand, using a
Taylor expansion, we get
\begin{equation}\label{tfi1}
H_n(h_n)-H_n(0)=h_nH'_n(0)+\frac{h_n^2}2H''_n(0)+\frac{h_n^3}6H_n^{(3)}(z_n),\qquad z_n\in[0,h_n].
\end{equation}

Hence (\ref{tfi1}) becomes
\begin{equation}
H_n(h_n)-H_n(0)=-\frac{h_n^2\alpha^2_{\tau_n}}2+\frac{h_n^3}6H_n^{(3)}(z_n),\ z_n\in[0,h_n].
\end{equation}
Now the expression of $H^{(3)}_n$ is a rational fraction of some partial derivatives of \\$ \E\left(e^{\xi X^{(n)}+h(Y^{(n)}-\EMD)}\right).$ The denominator of this rational fraction is bounded away from $0$ as it converges to a variance and numerator is bounded by (\ref{hyp3}). Hence  $H^{(3)}_n$ is bounded and the claim is proved.\QED
\section{Examples}\label{discution}
\begin{comment}
\subsection{Discussion}In the present paper, we have assumed that $p_n,\, p$ and
$q_n$ belong to $\mathbb{N}^*$. Indeed, if $p_n=0$, then the change of
probability (\ref{change1}) (see Subsection \ref{newprob}) implies that
$X^{(n)}$ is almost surely equal to $0$. 

The convex hull of the support of
$X$ is  included in $[0,+\infty[$, hence $0$ does not belong to its
    interior. Therefore we exclude the case $p=0$. 

The proofs of the deviation principles (Theorems \ref{} and \ref{}) make use of Lemma \ref{lempertub},
in which a factor $1/\sqrt{nq_n}$ appears. 
Therefore we exclude the case $q_n=0$. 

\subsection{Examples}
\end{comment}
In this section we give two examples of applications and one counter example. These  examples are borrowed from \cite{Ja01}.    
\subsubsection{Occupancy problem}\label{exocc}
In the classical occupancy problem (see \cite{Ja01} and the references therein for more details), $m$ balls are distributed at random into $N$ urns. The resulting numbers of balls $Z_1,\cdots,Z_N$ have a multinomial distribution, and it is well-known that this equals the distribution of $\blp X_1,\cdots,X_N\brp$ conditioned on $\sum_{i=1}^N X_i=m$, where $ X_1,\cdots,X_N$ are i.i.d. with $X_i\sim\mathcal{P}(\lambda),$ for an arbitrary $\lambda>0$. The classical occupancy problem studies the number $W$ of empty urns; this is thus $\sum_{i=1}^N \ind{\{X_i=0\}}$ conditioned on $\sum_{i=1}^N X_i=m$.\\
Now suppose that $m=np_n\to\infty$ and $N=nq_n\to\infty$ with $\frac{p_n}{q_n}\to\frac{p}q$. Take $X_i^{(n)}\sim\mathcal{P}(\lambda_n)$. Note that we do not assume that $\lambda_n=p_n/q_n$ and $\lambda=p/q$ which is the case in Janson's work. 
It is easy to see that Assumption 3. of Theorem \ref{thpgd} is fulfilled and that $\psi_{X,Y}$ is essentially smooth. Moreover, for $(x,y)\in\R^2$ and $(p,q)\in(\R^*)^2$, we have
\begin{align*}
\psi_X(x)&=-\lambda+\lambda e^x,\\
\psi^*_X(p/q&)=p/q\log(\frac{p}{q\lambda})+\lambda-\frac{p}{q},\\
\psi_{X,Y}(x,y)&=-\lambda+\log\Blp e^{\lambda\exp(x)	}-1+e^y\Brp.
\end{align*}
Hence we can apply Theorem \ref{thpgd}. Here the function $\psi^*_{X,Y}$
does not have any explicit form. We give in Appendix the graph of the rate
function for some particular values of $p/q $ and $\lambda$. Assumptions
of Theorem \ref{modev} are obviously fulfilled. We have
\begin{align*}
\E\blp \check{Y}^{(n)}_{\tau_n}\brp&=e^{-\lambda_n\exp(\tau_n)},\\
\P(\check{X_\tau}=k)&=e^{-\lambda\exp(\tau)}(e^{\tau}\lambda)^k/k!.
\end{align*}
Hence $\check{X_\tau}$ is Poisson with parameter $\lambda e^\tau$.  An easy calculation gives
\begin{align*}
\mathrm{cov}\blp\check{X}_\tau,\check{Y}_\tau\brp&=-\lambda e^\tau e^{-\lambda\exp(\tau)},\\
\mathrm{Var}(\check{Y}_\tau)&=e^{-\lambda\exp(\tau)}(1-e^{-\lambda\exp(\tau)}).
\end{align*}
Hence $\alpha^2_\tau=e^{-\lambda\exp(\tau)}\Blp
1-e^{-\lambda\exp(\tau)}+\lambda e^\tau e^{-\lambda\exp(\tau)}\Brp.$ Now, as
$\tau=\log(\frac{p}{q\lambda})$, we get $J(.)=\frac{(.)^2e^\lambda}{ 1-e^{-\lambda}+\lambda
  e^{-\lambda}}$ in the particular case where $\lambda=p/q$. Note that functional L.D.P. is given in \cite{BLG}.
\begin{rem}
Theorem \ref{thpgd} allows us to deal with other statistics than
$\sum_{i=1}^n \ind{\{X_i=0\}}$. For example, statistics of the form
$\sum_{i=1}^n f(X_i,Z_i)$ where $Z_1, \ldots , Z_n$ are i.i.d. and
independent from $X_1, \ldots , X_n$. Let us describe the particular case of bootstrap see \cite{Efron}.
Let $Z_1, \ldots , Z_n$ be i.i.d. real valued random variables, independent
of $X_1, \ldots , X_n$. We choose at
random with replacement a sample
$Z_1^*, \ldots, Z_n^*$.
Then $\sum_{i=1}^n f(Z_i^*)$ is distributed as $\sum_{i=1}^n X_i f(Z_i)$
conditioned on $\sum_{i=1}^n X_i=m$, where $X_i \sim \mathcal{P}(\lambda) \,
, \; i=1, \ldots , n$ for any $\lambda >0$. Hence we get the same kind of
conditioning as for the occupancy problem.
\end{rem}
\subsubsection{Branching processes}
Consider a Galton-Watson process, beginning with one individual, where the
number of children of an individual is given by a random variable $X$ having
finite moments. Assume further that $ \E(X)=1$. We number the individuals as
they appear. Let $X_i$ be the number of children of the $i-$th individual. It is well known (see example 3.4 in \cite{Ja01} and the references therein) that the total progeny is $n\geq 1$ if and only if
\begin{equation}
S_k:=\sum_{i=1}^kX_i\geq k\mbox{ for } 0\leq k<n \mbox{ but }S_n=n-1 \, .
\label{GW1}
\end{equation}
This type of conditioning is different from the one studied in the present
paper, but
Janson proves \cite[Example 3.4]{Ja01} that if we ignore the order of $X_1,
\ldots, X_n$, conditioning on (\ref{GW1}) is equivalent to conditioning
on $S_n=n-1$.
Hence we can
study variables of the kind $Y_i=f(X_i)$. Considering the case where $Y_i=\ind{\{X_i=3\}}$,
the $\sum_{i=1}^nY_i$ is the number of families with three children. Now
choosing $X_i\sim\mathcal{P}(\lambda)$, we compute the rate function as in
Example \ref{exocc}.
\subsubsection{Hashing}
\begin{bf}The model\end{bf}\\
Hashing with linear probing can be regarded as throwing $n$ balls
sequentially into $m$ urns at random; the urns are arranged in a circle and a
ball that lands in an occupied urn is moved to the next empty urn, always
moving in a fixed direction. The length of the move is called the
displacement of the ball, and we are interested in the sum of all
displacements which is a random variable noted $d_{m,n}$. We assume $n<m$.
\begin{comment}In order to make things clear we give an example. Assume $n=8$ and $m=10$
and that $(6,9,1,9,9,6,2,5)$ are the addresses where the balls want to
come. Let $d_i$ be the displacement of ball $i$, then $d_1=d_2=d_3=0$. Ball
number $4$ wants to land in urn $9$, but this urn is occupied by ball number
$2$, so she moves one step and lands in urn $10$ and $d_4=1$. Ball $5$ wants
to land in urn $9$, but she can't, so she moves to urn 10 which is also
occupied, so she moves to urn $1$, also occupied, so she moves to urn $2$ and $d_5=3$, and so on: $d_6=1,\ d_7=1,\ d_8=0$. The total displacement is here equal to $1+3+1+1=6.$
\end{comment}
\\
After throwing all balls, there are $N=m-n$ empty urns.  These divide the occupied urns into blocks of consecutive urns. For convenience, we consider the empty urn following a block as belonging to this block. 
\begin{comment}In our example there are two blocks, one contains urns $9,10,1,2,3$ (occupied) and urn $4$ empty, the other one urns $5,6,7$ (occupied) and $8$ empty.
\end{comment}
Janson \cite{Ja011} proved that the length of the blocks (counting the
empty urn) and the sum of displacements inside each block are distributed as
$(X_1,Y_1),\ldots,(X_N,Y_N)$ ($N=m-n$) conditioned on $\sum_{i=1}^NX_i=m$, where $(X_i,Y_i)$ are i.i.d. copies of
a pair $(X,Y)$ of random variables. $X$ has the Borel distribution
\begin{equation}
\P\blp X=l\brp=\frac{1}{T(\lambda)}\frac{l^{l-1}}{l!}\lambda^l,\quad
l\in\N^*,\ z\in[0,e^{-1}[ \, ,
\end{equation}
 where $T(\lambda)=\sum_{l=1}^\infty\frac{l^{l-1}}{l!}\lambda^l$ is the well-known tree function
 and $\lambda$ is an arbitrary number with $0<\lambda\leq e^{-1}$. The conditional distribution of $Y$ given $X=l$ is the same as the distribution of $d_{l,l-1}$.\\
Using Janson's results \cite{FPV97,Ja011,Ja01}, unfortunately we can prove that the joint
 Laplace transform of $(X_1,Y_1)$ is defined only on $(-\infty,a)\times(-\infty,0)$ for some positive $a$. Hence our
 results can not be applied. Nevertheless, in a forthcoming work, we will study conditioned L.D.P for self-normalized sums in the spirit of \cite{Shao03}. In that case the Laplace will be defined.
\subsubsection{Bose-Einstein statistics}
This example is borrowed from \cite{Holst79}. Consider $N$ urns. Put $n$ indistinguishable balls in the urns in such a way that each distinguishable outcome has the same probability i.e.,
$$
1/ \begin{pmatrix}n+N-1\\n\end{pmatrix},
$$
see for example \cite{Feller68}. Let $Z_k$ be the number of balls in the
$k$th urn. It is well known that $(Z_1,\ldots,Z_N )$ is distributed as $\blp
X_1,\cdots,X_N\brp$ conditioned on $\sum_{i=1}^N X_i=n$, where $
X_1,\cdots,X_N$ are i.i.d. with a geometric distribution. As for Example
\ref{exocc}, we can get a L.D.P for variables of the form $\sum_i h(X_i)$ if
$\mbox{dom}\psi_{h(X_i)}=\R$.
\subsubsection{Possible extensions}
\begin{comment}
 The quantities of interest are then of the form   $\sum_i h_i(X_i)$ indeed if $n$ and $N-1$ correspond to the sizes of two samples. The rank of the "second" sample can be written $R_j=j+\sum{k=1}^jZ_k$ so $\sum_{k=1}^NkU_k$ is a linear transformation of the Wilcoxon statistic. 
In this paper we consider the law of $T_{n}=Y^{(n)}_{1}+\cdots+Y^{(n)}_{nq_n}$ conditioned on
$S_{n}=X^{(n)}_{1}+\cdots+X^{(n)}_{nq_n}=np_n$. Our assumptions are  that
the couple $\left(X^{(n)}_{i},X^{(n)}_{i}\right)_i$ are i.i.d. random
variables and that $(p_n,q_n)\in\N^2$ Let us  enumerate some of the possible
extensions
\end{comment}
Among possible extensions, let us mention the case where the variables $Y_i$
are independent but do not have the same distribution. 
This case occurs in \cite[Examples 2 and 3]{Holst79}, where the quantity of
interest is the law of $\sum_{i=1}^Nh_i(X_i)$ conditioned on the event
$\sum_{i=1}^NX_i=n$. Another way to extend our work is to deal with the case
where the variables $X_i$ are independent but not i.i.d.. This case occurs when counting from a random permutation the number of cycles of a fixed size see for example \cite[Chapter 1]{ABT03}

In the present paper we assume that $p_n,q_n,p$ and $q$ are positive. The
other cases will be considered in a forthcoming work.
\begin{comment}
\item Assume only that the r.v. $Y$'s are independent but not having the same law . This case occurs in the following example. Consider a permutation of $\{1,\ldots,n\}$ chosen uniformly and at random from  the $n!$ possible permutations in $\mathcal{S}_n$, let $C_j{(n)}$ be the number of cycle of length $j$ then it is known that see \cite{ABT03} Chapter one that $\mathcal{L}\blp C_1{(n)},\ldots,C_n{(n)}\brp=\mathcal{L}\blp Z_1,\ldots,Z_n\brp| Z_1+2Z_2+\ldots+nZ_n=n$. Hence 
\item case $p_n, q_n\to0,\infty$
\end{enumerate}
\end{comment}

\section{Appendix}\label{app}
Here we give the shape of the large deviation rate function 
in the example of the Occupancy problem when $\lambda=1$ and $p/q=1$, $p/q=0.4$ or $p/q=3$

\begin{center}
\begin{figure}[hb]
\psfrag{A}{Rate function for $\lambda=1$, $p/q=1$}
\psfrag{B}{Rate function for $\lambda=1$, $p/q=3$}
\psfrag{C}{Rate function for $\lambda=1$, $p/q=0.4$}
\includegraphics[height=7cm]{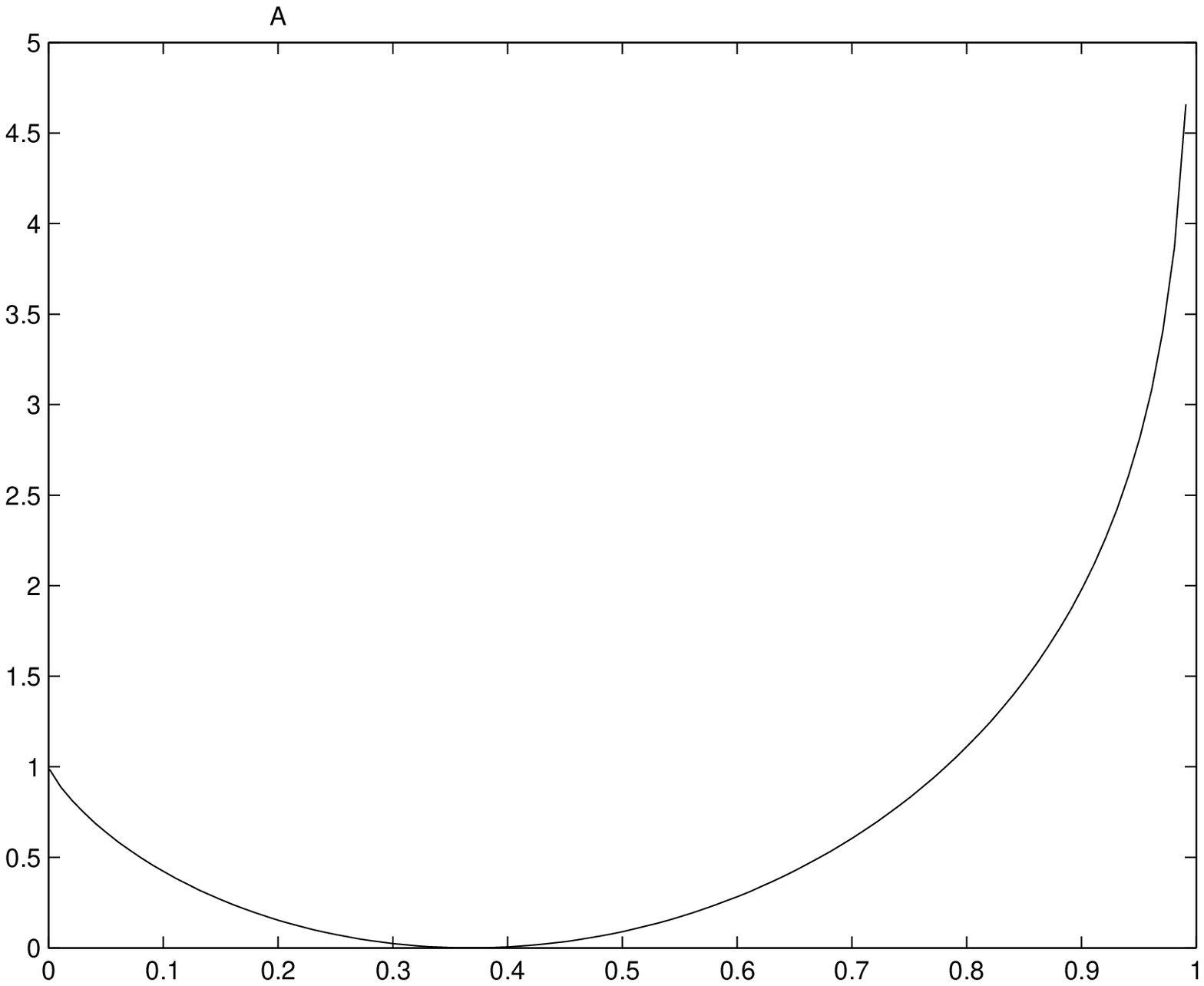} 
\includegraphics[height=7cm]{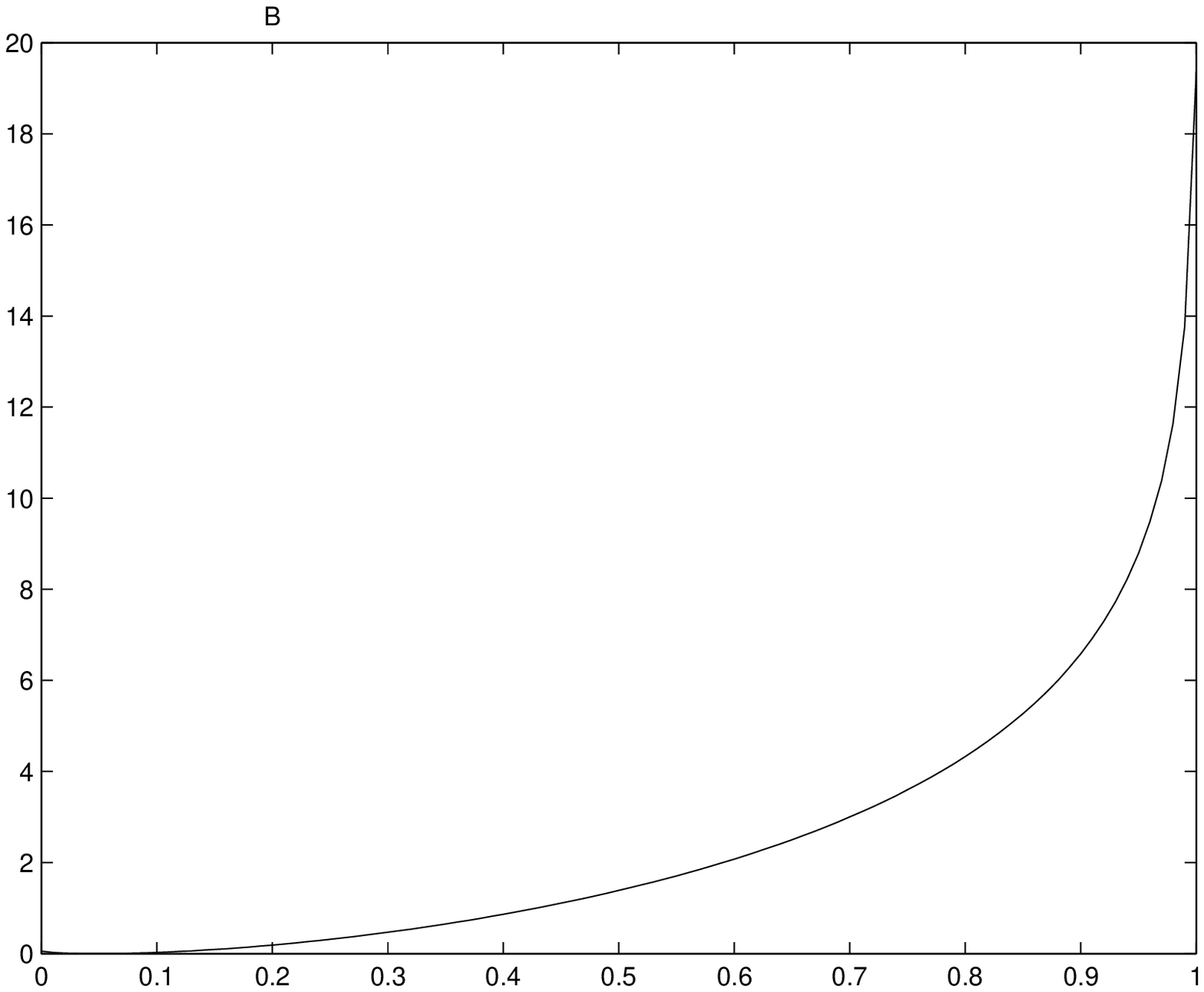}
\includegraphics[height=7cm]{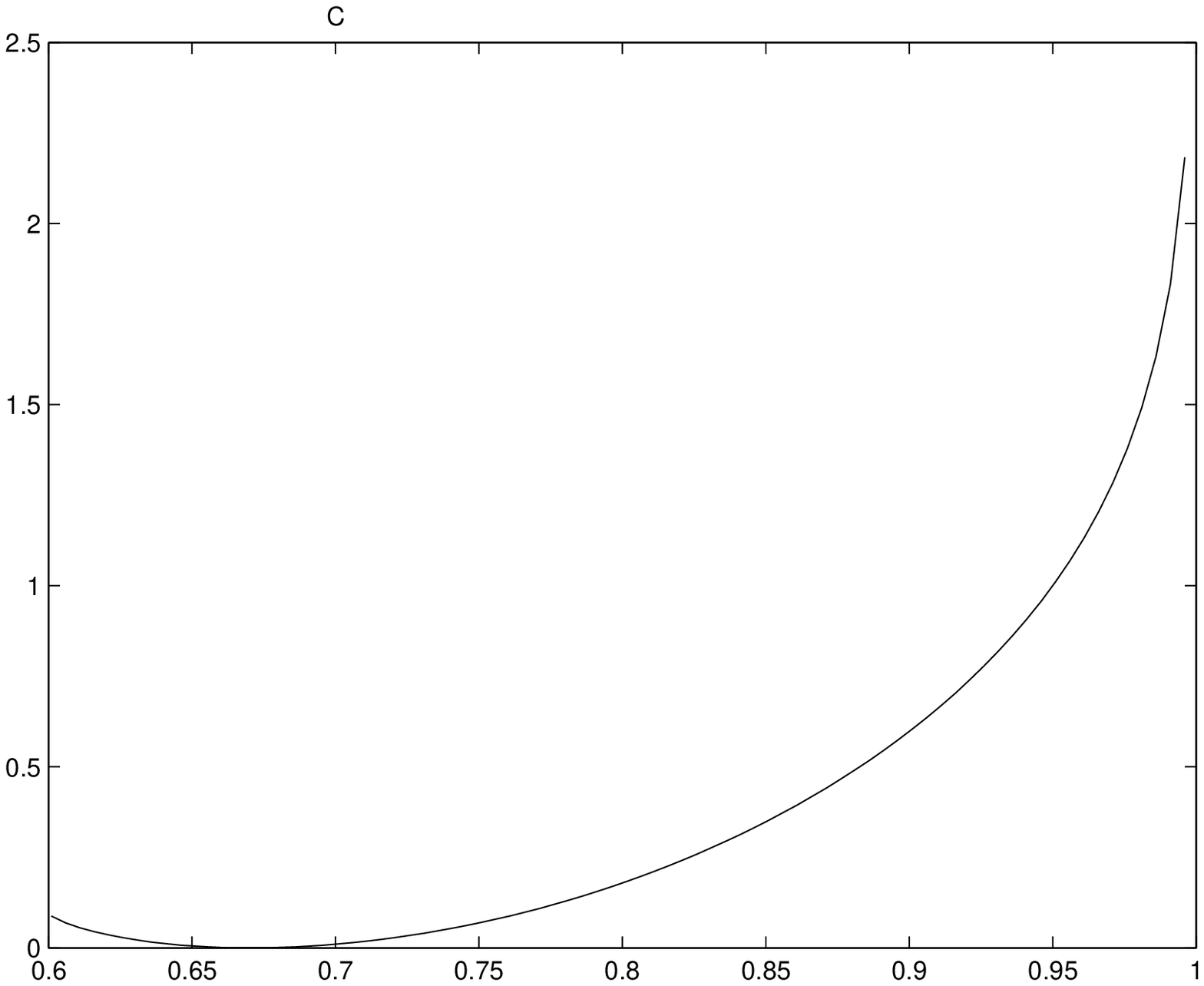}
\end{figure}
\end{center}

\end{document}